\documentclass[12pt,a4paper,reqno]{article}

\usepackage[a4paper,width=160mm,top=25mm,bottom=30mm]{geometry}
\usepackage{hyperref}
\usepackage[utf8]{inputenc}
\usepackage[T1]{fontenc}
\usepackage[english]{babel}
\usepackage{xcolor}
\usepackage{float}
\usepackage{stmaryrd}

\usepackage{amsthm}

\theoremstyle{plain}
\newtheorem{theorem}{Theorem}[section]

\newtheorem{proposition}[theorem]{Proposition}

\newtheorem{lemma}[theorem]{Lemma}

\theoremstyle{remark} 
\newtheorem*{remark}{Remark}

\theoremstyle{definition} 
\newtheorem*{definition}{Definition}

\usepackage{amsmath}
\usepackage{amssymb}
\usepackage{dsfont} 
\usepackage{mathtools}
\mathtoolsset{showonlyrefs}

\usepackage{MnSymbol} 
\usepackage{calc} 

\newlength{\arrow}
\usepackage{xargs}

\usepackage{tabularx}
\usepackage{graphicx}
\usepackage{caption}
\begin{document}
\title{Locality of percolation for graphs with  polynomial growth}
\date{\today} 

\author{Daniel Contreras\thanks{ETH Zürich, Rämistrasse 101, 8092 Zurich Switzerland} 
\and
Sébastien Martineau\thanks{LPSM, Sorbonne Université, 4 place Jussieu, 75005 Paris France}
\and
Vincent Tassion\footnotemark[1]} 

\newcommand{\defini}{\textbf}


\maketitle

\begin{abstract}
Schramm's Locality Conjecture asserts that the value of the critical percolation parameter~$p_c$ of a graph satisfying $p_c<1$ depends only on its local structure.
In this note, we prove this conjecture in the particular case of transitive graphs with polynomial growth.
Our proof relies on two recent works about such graphs, namely supercritical sharpness of percolation by the same authors and a finitary structure theorem by Tessera and Tointon.
\end{abstract}

\section{Introduction}

Around 2008, Schramm conjectured that, under some non-degeneracy assumption, the value of the critical probability for percolation depends only on the local structure of the underlying graph.  This means that two transitive graphs with similar local structure should have close critical probabilities.

 Let us recall a  formal version of this conjecture. In this paper, graphs are taken to be simple, non-empty, locally finite, and connected.
Given two transitive graphs $\mathcal{G}$ and $\mathcal{H}$, define
\begin{equation*}
	R(\mathcal{G},\mathcal{H}):=\max\{k\in\mathbb{N}\cup\{\infty\}:B_\mathcal{G}(k)\simeq B_\mathcal{H}(k)\}.
\end{equation*}
Let $(\mathcal{G}_k)$ be a sequence of transitive graphs and let $\mathcal{G}_\infty$ also be some transitive graph. Say that $(\mathcal{G}_k)$ \defini{converges to $\mathcal{G}_\infty$ (for the local topology)} if $R(\mathcal{G}_k,\mathcal{G}_\infty)$ converges to infinity. Mathematically, we write $\mathcal{G}_k \xrightarrow[k\to\infty]{}\mathcal{G}_\infty$.

Schramm's Locality Conjecture is the following statement. Let $(\mathcal{G}_k)$ be a sequence of transitive graphs such that $p_c(\mathcal{G}_k)<1$ for every $k$. Assume that $\mathcal{G}_k \xrightarrow[k\to\infty]{}\mathcal{G}_\infty$ for some transitive graph $\mathcal{G}_\infty$.
Then $p_c(\mathcal{G}_k) \xrightarrow[k\to\infty]{}p_c(\mathcal{G}_\infty)$.

This conjecture first appeared in~\cite{benjamini2011critical}, where  the authors tackled the case of a sequence of uniformly non-amenable graphs that converges locally to an infinite regular tree. Since then, the conjecture has been established for Cayley graphs of Abelian groups in~\cite{MR3630298} and for graphs with uniform exponential growth in  \cite{2018arXiv180808940H}.

Notice that the assumption with $p_c<1$ cannot be removed. Indeed, for the usual Cayley graph structures, both sequences $(\mathbb{Z}/k\mathbb{Z})^2$ and $\mathbb{Z}\times(\mathbb{Z}/k\mathbb{Z})$ converge to $\mathbb{Z}^2$. However, all graphs of these sequences satisfy $p_c=1$ while the square lattice has $p_c<1$.

The conjecture was originally stated under the more stringent condition $\sup_k p_c(\mathcal{G}_k) < 1$ but it results from \cite{hutchcroft2021non} that if $(\mathcal{G}_k)$ converges for the local topology and satisfies $p_c(\mathcal{G}_k)<1$ for every $k$, then it automatically satisfies the seemingly stronger condition $\sup_k p_c(\mathcal{G}_k) < 1$. Panagiotis and Severo have proved in \cite{panagiotisevero} that, for Cayley graphs, the previous sentence is correct even without assuming that the sequence $(\mathcal{G}_k)$ converges to some transitive graph. Their result has been made quantitative and explicit in \cite{lmtt}.

\paragraph{Locality for graphs of polynomial growth}

In this paper, we establish Schramm's Locality Conjecture under the assumption that $\mathcal{G}_\infty$ has polynomial growth, i.e.~that the cardinality the ball of radius $n$ is upper-bounded by a polynomial in $n$. This article can be read with $p_c$ meaning either always $p_c^\mathsf{site}$ or always $p_c^\mathsf{bond}$, both interpretations yield correct statements.

\begin{theorem}\label{thm}
Let $(\mathcal{G}_k)_{k\in \mathbb N}$ be a sequence of transitive graphs such that $p_c(\mathcal{G}_k)<1$ for every $k\in \mathbb N$.  Let  $\mathcal{G}_\infty$  be a transitive graph of polynomial growth. If  $\mathcal{G}_k \xrightarrow[k\to\infty]{}\mathcal{G}_\infty$, then 
\begin{equation}
p_c(\mathcal{G}_k) \xrightarrow[k\to\infty]{}p_c(\mathcal{G}_\infty).\label{eq:2}
\end{equation}
\end{theorem}

\begin{remark} Only the graph $\mathcal G_\infty$ is assumed to have polynomial growth. Actually, it follows from \cite{tesseratointon} that if $\mathcal{G}_\infty$ has polynomial growth and $\mathcal G_k\xrightarrow[k\to\infty]{}\mathcal G_\infty$, then $\mathcal G_k$ also has polynomial growth for all $k$ large enough.
\end{remark}

Since the inequality $\limsup p_c(\mathcal{G}_k)\le p_c(\mathcal{G}_\infty)$ is known in full generality (see~\cite[Section 14.2]{pete2014probability} or \cite[Section 1.2]{duminil2016newproof}), we only need to take $p> p_c(\mathcal{G}_\infty)$ and prove that $p\ge p_c(\mathcal{G}_k)$ for $k$ large enough.
To do so, we use supercritical sharpness on the limit graph $\mathcal{G}_\infty$ to build a finite-size event that has good probability to occur and that, when occurring, guarantees good connections within some box. This first step relies on \cite{supercriticalCMT2021}. Then, we use finitary structure theorems for the geometry of transitive graphs of polynomial growth in order to perform a renormalisation argument. This enables us to go from local to global: from the fact that our finite-size event holds with good probability, we obtain that there is an infinite cluster in $\mathcal{G}_k$ at parameter $p$. This second step relies on \cite{tesseratointon} and ``finitary'' means that global geometric information can be deduced from suitable information inside a \emph{single} ball.

\paragraph{Uniformly controlled nets}

A key point in the proof is that the aforementioned renormalisation on $\mathcal G_k$ can be performed uniformly with respect to $k$: we need the scale of renormalisation not to depend on the graph $\mathcal G_k$.      
To achieve this, we use Proposition~\ref{prop} below and   rely on the standard notion of net, which we now define.

Given two numbers $a,b\ge 1$, an  $(a,b)$-net of a graph $(V,E)$ is a subset that is $a$-separated and $b$-dense. Namely, it is some $V_0 \subset V$ such that $$\min \{d(x,y) : x,y\in V_0, x\neq y\}\geq a~~~~~\text{and}~~~~~\max\{d(x,y) : x\in V_0, y \in V\}\leq b.$$ An $(a,b)$-net has a natural graph structure:  two distinct elements $x$ and $y$ of $V_0$ are said to  be adjacent if $d(x,y)\leq 4b$. This graph structure depends on $V_0$ but also on the particular $b$ we have in mind when considering $V_0$ as an $(a,b)$-net. When some $V_0$ is considered as an $(a,b)$-net, we define $p_c^{\mathrm{site}}(V_0)$ as the critical parameter for Bernoulli site percolation on $V_0$, equipped with this graph structure.

\begin{definition}
  \label{def:1}
  Let $\mathbb G$ be a collection of transitive graphs and let $C\ge 1$. We say that $\mathbb G$  has \defini{$C$-controlled nets} if  for every $a\ge 1$,  every graph   $\mathcal{G}\in \mathbb G$  admits an $(a,Ca)$-net $V_0$ with $p_c^{\mathrm{site}}(V_0)\le\tfrac{3}{4}$.
\end{definition}

In the definition above, the value $3/4$ does not play a specific role. Any fixed value $\alpha<1$  would work equally well for our purpose. The value  $3/4$ appears in Lemma~\ref{lem:zdeux} as an upper bound for the critical parameter of Bernoulli site percolation on $\mathbb Z^2$.

  \begin{proposition}\label{prop}
Let $(\mathcal{G}_k)$ be a sequence of transitive graphs such that $\forall k\in\mathbb{N},~p_c(\mathcal{G}_k)<1$. If  $\mathcal{G}_k \xrightarrow[k\to\infty]{}\mathcal{G}_\infty$ for some transitive graph $\mathcal{G}_\infty$ of polynomial growth, then there is some constant  $C\ge 1$ such that the collection $\{\mathcal{G}_k:k\ge C\}$  has $C$-controlled nets. 
\end{proposition}

\subsection*{Organisation of the paper}

The proof of Proposition~\ref{prop} is presented in Section~\ref{sec:proof-proposition}.  In Section~\ref{sec:proof-theorem-refthm}, we deduce Theorem~\ref{thm} from Proposition~\ref{prop} by using a standard finite-size criterion approach. Our proof of Proposition~\ref{prop} uses a generalisation of the monotonicity result of Benjamini and Schramm \cite[Theorem 1.1]{MR1423907}, which states that the critical parameter  $p_c$ of   a
graph is always greater than or equal to the critical parameter of any of its covering graphs. For completeness, we present this generalised statement in Section~\ref{sec:monot-p_c-revis}.

\subsection*{Acknowledgments}
\label{sec:acknoledgment}

We thank Romain Tessera and Matthew Tointon for helpful comments regarding nilpotent geometry.

The first and third authors are supported by the European Research Council (ERC) under the European Union’s Horizon 2020 research and innovation program (grant agreement No 851565) and by the NCCR SwissMap.

\section{Uniformly controlled nets for converging sequences}
\label{sec:proof-proposition}

We want to prove that if $(\mathcal{G}_k)$ is a sequence as in Proposition~\ref{prop}, then, for some $C$, the collection $\{\mathcal{G}_k:k\ge C\}$  has $C$-controlled nets. In order to prove this statement, we first prove that the collection of all Cayley graphs of $\mathbb{Z}^2$   has $C$-controlled nets. Then, we extend this result to all Cayley graphs of nilpotent groups satisfying $p_c<1$. Finally, by using \cite{tesseratointon}, we obtain the desired statement.

\subsection{Cayley graphs of $\mathbb{Z}^2$  have  controlled nets}

Given a finite generating subset $S$ of $\mathbb{Z}^2$ and some $v\in\mathbb{Z}^2$, we write $\|v\|_S$ for the distance between 0 and $v$ in the Cayley graph $\mathsf{Cay}(\mathbb{Z}^2,S)$.

On an Abelian Cayley graph, distances can be well analysed by elementary linear algebra. Here, we use the following lemma, which provides a useful control on the distances in Cayley graphs of $\mathbb Z^2$.

\begin{lemma}
\label{lem:1}
For every finite generating subset $S$ of $\mathbb{Z}^2$, there are $u,v\in S$ such that
$$
\forall m,n\in \mathbb{Z},\qquad \tfrac{|m|+|n|}{3}\leq \|mu+nv\|_S\leq |m|+|n|.
$$
\end{lemma}

\begin{proof}
  See $\mathbb{Z}^2$ as a subset of $\mathbb{R}^2$ endowed with the usual Euclidean norm $\|\cdot\|_2$. Pick $u$ an element of $S$ with maximal Euclidean norm. Denote by $p$ the linear orthogonal projection on $u^\perp$. Pick an element $v$ of $S$ that maximises $\|p(v)\|_2$.

Let $m,n\in \mathbb{Z}$. By the triangle inequality in $\mathsf{Cay}(\mathbb{Z}^2,S)$, we have $$\|mu+nv\|_S\leq |m|\cdot \|u\|_S+|n|\cdot\|v\|_S=|m|+|n|,$$ which would hold for any choice of $u$ and $v$ in $S$.
As $u$ maximises its Euclidean norm inside $S$, by the triangle inequality in $(\mathbb{R}^2,\|\cdot\|_2)$, we have
\[(|m|-|n|)\|u\|_2\le|m|\cdot\|u\|_2-|n|\cdot\|v\|_2\le\|mu+nv\|_2\leq \|mu+nv\|_S \|u\|_2,\]
whence $\|mu+nv\|_S  \ge |m|-|n|$. On the other hand, as $v$ maximises $\|p(v)\|_2$ in $S$, the triangle inequality in $(\mathbb{R}^2,\|\cdot\|_2)$ yields
\[|n|\cdot\|p(v)\|_2=\|p(nv)\|_2=\|p(mu+nv)\|_2\le \|mu+nv\|_S\cdot\|p(v)\|_2,\]
 whence $\|mu+nv\|_S\ge |n|$. We conclude by observing that $\max(|m|-|n|,|n|)\ge\tfrac{|m|+|n|}{3}$.
\end{proof}

\begin{lemma}
\label{lem:zdeux}
The collection of all Cayley graphs of $\mathbb{Z}^2$  has $1$-uniformly controlled nets.
\end{lemma}

\begin{proof}
Let $S$ be a finite generating subset of $\mathbb{Z}^2$. Let $u$ and $v$ be such that the conclusion of Lemma~\ref{lem:1} holds. Let $a\ge 1$ and $m=\lceil 3a \rceil$. Let $\Gamma$ denote the subgroup of $\mathbb{Z}^2$ generated by $mu$ and $mv$. By the choice of $u$ and $v$, this defines an $a$-separated subset of $\mathsf{Cay}(\mathbb{Z}^2,S)$. Let $V_0$ be a maximal $a$-separated subset of $\mathsf{Cay}(\mathbb{Z}^2,S)$ containing $\Gamma$, maximality being understood relative to inclusion. By maximality, $V_0$ is an $(a,a)$-net. Furthermore, via the embedding $(k,\ell)\mapsto kmu+\ell mv$, this net contains the square lattice as a subgraph. It therefore satisfies $p_c^\mathrm{site}\le \tfrac{3}{4}$.
\end{proof}

\subsection{Nilpotent groups  have  controlled nets}

Lemma~\ref{lem:zdeux} gives us some uniform control over all Cayley graphs of $\mathbb{Z}^2$. Combined with the observation that any nilpotent group with $p_c<1$ admits $\mathbb{Z}^2$ as a quotient, this allows us to get the following statement.

\begin{lemma}
\label{lem:2}
The collection of all Cayley graphs of nilpotent groups with $p_c<1$ has  $2$-controlled nets.
\end{lemma}

\begin{remark}
If the Cayley graph $\mathcal{G}=\mathsf{Cay}(G,S)$  under consideration was the product of a Cayley graph of $\mathbb{Z}^2$ and another Cayley graph $\mathcal{H}$, it would suffice for our purpose to take a good $(a,a)$-net $V_1$ of $\mathbb{Z}^2$, an arbitrary $(a,a)$-net $V_2$ of $\mathcal{H}$, and to prove that $V_0=V_1\times V_2$ is a suitable net because of the choice of $V_1$. Here, $\mathcal{G}$ does not necessarily split as a product but it will still be possible to produce a suitable net, by using the fact that $\mathbb{Z}^2$ is a quotient of $G$. 
\end{remark}

\begin{proof}
Let $\mathcal{G}=\mathsf{Cay}(G,S)$ be a Cayley graph where the group $G$ is nilpotent and assume that $p_c(\mathcal{G})<1$. It is well-known that it is possible to fix a surjective group homomorphism $\pi$ from $G$ to $\mathbb{Z}^2$ (see e.g. \cite[Lemma 3.23]{hutchcroft2021non}).
Let $a\ge 1$. By Lemma~\ref{lem:zdeux}, we can pick $V_1$ an $(a,a)$-net of $\mathsf{Cay}(\mathbb{Z}^2,\pi(S))$ with $p_c^\mathrm{site}<\tfrac{3}{4}$. For each $x\in V_1$, pick a maximal $a$-separated\footnote{Given two distinct points in the set, any path connecting them has length at least $a$. The paths are allowed to exit the fibre $\pi^{-1}(\{x\})$.} subset of $\pi^{-1}(\{x\})$ and denote it by $U_x$. Let $V_0:=\bigcup_{x\in V_1} U_x$. We shall prove that $V_0$ is an $(a,5a)$-net of $\mathcal{G}$ with $p_c^\mathrm{site}\le\tfrac{3}{4}$.

First, observe that the set $V_0$ is $a$-separated in $\mathcal{G}=\mathsf{Cay}(G,S)$. Indeed, let $g$ and $h$ be two distinct points in $V_0$. We have $d_{S}(g,h)\ge d_{\pi(S)}(\pi(g),\pi(h))$. Thus, if $\pi(g)\neq \pi(h)$, we are done by definition of $V_1$. Otherwise, $g$ and $h$ belong to $\pi^{-1}(\{x\})$ for $x=\pi(g)=\pi(h)$, and then it holds by definition of $U_x$.

Let us now prove that $V_0$ is $2a$-dense in $\mathsf{Cay}(G,S)$. 
Let $g\in G$ and $x=\pi(g)$. Since $\pi$ is a quotient map and $V_1$ is $a$-dense in $\mathsf{Cay}(\mathbb{Z}^2,\pi(S))$, we can pick some $h\in \pi^{-1}(\{x\})$ such that $d(g,h)\le a$.
Since $U_x$ is maximal as an $a$-separated set, we have $d_S(h,U_x)<a$. As a result, $d(g,U_x)<2a$.

Finally, we show that $V_0$, considered with its graph structure of $(a,2a)$-net, has $p_c^\mathrm{site}\le \tfrac{3}{4}$. The map $\pi$ is well-defined seen from $V_0$ to $V_1$. When $V_0$ and $V_1$ are viewed as graphs\footnote{respectively as an $(a,2a)$-net of $\mathcal{G}$ and a $(a,a)$-net of $\mathsf{Cay}(\mathbb{Z}^2,\pi(S))$}, this map between vertex-sets satisfies the following properties:
\begin{itemize}
\item $\pi : V_0\to V_1$ is surjective,
\item for every $x\in V_1$, every $g\in U_x$ and every $V_1$-neighbour $y$ of $x$, there is a $V_0$-neighbour $h$ of $g$ such that $\pi(g)=y$.
\end{itemize}
Indeed, the same proof as that of $2a$-density yields an $h$ such that $$d_S(g,h)< 4\times a+a< 4\times 2 a.$$
By the forthcoming Proposition~\ref{prop:mono}, the existence of such a map $\pi: V_0\to V_1$ implies that $p_c^\mathrm{site}(V_0)\le p_c^\mathrm{site}(V_1)\leq \tfrac{3}{4}$.
\end{proof}

\subsection{Converging sequences have  controlled nets}

We are now able to prove Proposition~\ref{prop}.
Let us take $(\mathcal{G}_k)$ a sequence of transitive graphs such $\forall k,~p_c(\mathcal{G}_k)<1$. Assume that $\mathcal{G}_k$ converges locally to some transitive graph $\mathcal{G}_\infty$ of polynomial growth. We prove that there exists a constant $C$ such that $\{\mathcal G_k, \,k\ge C\}$ has $C$-controlled nets.

By \cite[Theorem 1.3]{tesseratointon}, up to forgetting finitely many terms of the sequence $(\mathcal{G}_k)$, we can do the following:
\begin{enumerate}
\item we assume that all $\mathcal{G}_k$'s have polynomial growth,
\item \label{item:deux} we  fix some constant $A\ge 1$, Cayley graphs of nilpotent groups $\mathsf{Cay}(G_k,S_k)$ and maps $\varphi_k : G_k \to V(\mathcal{G}_k)$ such that every $\varphi_k$ is an $A$-quasi-isometry.
\end{enumerate}
Recall that $\varphi_k$ being an $A$-quasi-isometry means that
$$
\forall g,h\in G_k, \qquad \tfrac{1}{A}d(\varphi_k(g),\varphi_k(h))-A\le d(g,h)\le Ad(\varphi_k(g),\varphi_k(h))+A
$$
and $\varphi_k(G_k)$ is $A$-dense in $\mathcal{G}_k$.

The fact that $p_c$ is smaller than 1 is preserved by quasi-isometries: see \cite[Theorem 7.15]{lyonsperes17} and the comment following its proof. Therefore, for every $k$, the group $G_k$ is not finite or two-ended. Since $G_k$ is nilpotent, it thus has to be one-ended. As a result, Lemma~\ref{lem:2} applies to $\mathsf{Cay}(G_k,S_k)$.

Let $a\ge 1$ and $a'=aA+A^2$. By Lemma~\ref{lem:2}, we can fix some $(a',2a')$-net $V_k$ in $\mathsf{Cay}(G_k,S_k)$ satisfying $p_c^\mathrm{site}\le \tfrac{3}{4}$. Setting $V_k'=\varphi_k(V_k)$ thus defines an $(a,2a'A+A)$-net of $\mathcal{G}_k$. Furthermore, the map $\varphi_k$ seen from $V_k$ to $V_k'$ is an injective graph homomorphism, when $V_k$ is considered as an $(a',2a')$-net and $V'_k$ as an $(a,2a'A+A)$-net. Therefore, we have $p_c^\mathrm{site}(V'_k)\le p_c^\mathrm{site}(V_k)\le \tfrac{3}{4}$. Taking $C=5A^3$ completes the proof of Proposition~\ref{prop}. \hfill \qed 

\vspace{0.4cm}

\begin{remark} Seeing why \cite[Theorem 1.3]{tesseratointon} indeed guarantees Item~\ref{item:deux} requires a basic tool of geometric group theory: the Milnor--Schwarz Lemma. We only need the following particular case: {\it For every $A\ge 1$, there is some constant $B$ such that the following holds. Let $G$ be a group generated by a finite subset $S$. Let $H$ be a subgroup of $G$ of index at most $A$. Then, $H$ admits a finite generating subset $S_H$ such that there is a $B$-quasi-isometry from $\mathsf{Cay}(H,S_H)$ to $\mathsf{Cay}(G,S)$.} This is proved by following the proofs of Theorem~8.37 and Corollary~8.47 in \cite{drutukapovich} and observing that $B$ depends only on $A$.
Tessera and Tointon did not state their Theorem 1.3 as our Item~\ref{item:deux} because they wanted a stronger result where, up to tolerating a bounded index subgroup, the multiplicative constant of the quasi-isometry can be taken equal to 1.
\end{remark}

\section{Proof of Theorem~\ref{thm}}
\label{sec:proof-theorem-refthm}

In this section, we use Proposition~\ref{prop} to perform a renormalisation argument yielding locality for transitive graphs of polynomial growth.

We first recall a useful lemma about $k$-independent percolation processes.
A site percolation process $\mathbb{P}$ is called \defini{$k$-independent} if, for any two sets of vertices $U_1$ and $U_2$ satisfying $$\min \{d(x_1,x_2):x_1\in U_1,~x_2\in U_2\}>k,$$ the restriction $X_{|U_1}$ is independent of $X_{|U_2}$, where $X$ denotes a $\mathbb{P}$-distributed random variable. Recall that $\mathbb{P}_1$ \defini{stochastically dominates} $\mathbb{P}_2$ if there is a coupling $(X_1,X_2)$ of $(\mathbb{P}_1,\mathbb{P}_2)$ such that every $X_2$-open site is always $X_1$-open.

\begin{lemma}
  \label{lem:3}
  Let $D$ and $k$ be two positive integers. There exists $q=q(k,D)<1$ such that the following holds.   For every connected graph $\mathcal{H}$ with maximal degree at most $D$, any $k$-independent site percolation process on $\mathcal{H}$ with marginals at least $q$ stochastically dominates independent site percolation on $\mathcal{H}$ of parameter $\tfrac{3}{4}$. 
\end{lemma}

\begin{proof}
    By \cite[Theorem~1.3]{liggett1997domination}, for every constant $D$, we can fix some constant $p_D<1$ such that for every connected graph $\mathcal{H}$ with maximal degree at most $D$, any $1$-independent site percolation process on $\mathcal{H}$ with marginals at least $p_D$ stochastically dominates independent site percolation on $\mathcal{H}$ of parameter $\tfrac{3}{4}$. Notice that if $\mathcal{H}$ has maximal degree at most $D$, then any $k$-independent site percolation process on $\mathcal{H}$ with marginals at least $p_{D^k}$ stochastically dominates independent site percolation on $\mathcal{H}$ of parameter $\tfrac{3}{4}$. Indeed, any $k$-independent site percolation process on $\mathcal{H}$ is $1$-independent when considered on $\mathcal{H}^{(k)}$, where $V(\mathcal{H}^{(k)})=V(\mathcal{H})$ and $E(\mathcal{H}^{(k)})=\{\{x,y\}~:~0<d_{\mathcal{H}}(x,y)\le k\}$. Setting $q(k,D)=p_{D^k}$ thus yields the lemma.
\end{proof}

We also recall a classical lower bound for distances in nets, which will help us prove that some auxiliary site percolation process is $80$-independent.

\begin{lemma}\label{lem:geom}
Let $V_0$ be an $(a,b)$-net of some graph $\mathcal{G}$.
Then, for any $u,v\in V_0$, we have:

$$d_{V_0}(u,v)\le \tfrac{1}{b}d_\mathcal{G}(u,v).$$
\end{lemma}

\begin{proof}
Let $u,v\in V_0$. Let $(u_0,\dots,u_\ell)$ be a shortest path from $u_0=u$ to $u_\ell=v$ in $\mathcal{G}$. Set $B=\lfloor 2b\rfloor\ge b$. By keeping only vertices of the form $u_{Bk}$ and the final vertex $u_\ell$, we obtain a sequence of vertices $(v_0,\dots,v_L)$ with $L\le\tfrac{\ell}{B}$ and such that $d_{\mathcal{G}}(v_i,v_{i+1})\le B$. Except for $v_0$ and $v_L$, these vertices have no reason to belong to $V_0$. Therefore, we define $(w_0,\dots,w_L)$ by setting $w_0=u$, $w_L=v$, and $w_i$ any vertex of $V_0$ such that $d_{\mathcal{G}}(v_i,w_i)\le b$ when $i\in\{1,\dots,L-1\}$. This sequence defines a path from $u$ to $v$ in $V_0$ and its length is at most $\tfrac{\ell}{b}=\tfrac{1}{b}d_\mathcal{G}(u,v)$.
\end{proof}

We are now able to prove Theorem~\ref{thm}.

\begin{proof}[Proof of Theorem~\ref{thm}]
Let  $\mathcal{G}_\infty$ be a transitive graph of polynomial growth.  Let $(\mathcal{G}_k)_{k\in \mathbb N}$ be a sequence of transitive graphs such that $p_c(\mathcal{G}_k)<1$ for every $k\in \mathbb N$ and  $\mathcal{G}_k \xrightarrow[k\to\infty]{}\mathcal{G}_\infty$.
Fix $p>p_c(\mathcal G_\infty)$. We will prove that for every $k$ large enough, there is an infinite cluster for percolation of parameter $p$ on $\mathcal{G}_k$ with positive probability.

\textbf{We first work on $\mathcal G_\infty$} and use that $p$ is supercritical to build a suitable finite-size event for percolation on  $\mathcal G_\infty$. For $v\in V(\mathcal{G}_\infty)$ and $n\ge 1$, define the event $\mathsf{E}_n(v)$ by the conjunction of the following two events for $p$-percolation on $\mathcal{G}_\infty$:
  \begin{itemize}
  \item there is a cluster intersecting $B_n(v)$ and touching the $10n$-sphere centred at $v$,
  \item any two open paths intersecting $B_{2n}(v)$ and touching the $5n$-sphere centred at $v$ are connected by an open path that lies within $B_{5n}(v)$.
  \end{itemize}
By \cite[Proposition~1.3]{supercriticalCMT2021}, for every $v$, the event $\mathsf{E}_n(v)$ has a probability converging to 1 when $n$ goes to infinity. Notice that, by transitivity, this probability depends on $n$ but not on $v$.

Let us quantify how large we need $\mathbb{P}(\mathsf{E}_n)$ to be.
By  \cite{bass72,guivarch73, gromov81, trofimov}, there exists an integer  $d\ge 2$ and a constant $c\ge 1$ such that the balls of $\mathcal G_\infty$ satisfy
\begin{equation}
  \label{eq:1}
  \tfrac1{c}r^d\le |B_r|\le cr^d
\end{equation}
for every $r\ge1$.  Let $C$  be such that the conclusion of Proposition~\ref{prop} holds. For reasons that will make sense shortly, define
\begin{equation}
  \label{eq:3}
  D=c^2\cdot  (12C+3)^d
\end{equation}
and fix some $n\ge 9C$ such that
\begin{equation}
\mathrm{P}_p(\mathsf{E}_n)\ge q(80,D).\label{eq:4}
\end{equation}
In this inequality, $q(80,D)$ is defined so that the conclusion of Lemma~\ref{lem:3} holds. Set  $a=\tfrac{n}{4C}$. For every $k\ge C$, fix $V_k$ an $(a,Ca)$-net of $\mathcal{G}_k$ such that $p_c^\mathrm{site}(V_k)\le \tfrac{3}{4}$. Observe that every $V_k$ has maximal degree at most $D$. Indeed, if some vertex $v$ has $V_k$-neighbours $v_1,\dots,v_m$, then the balls $B_{a/3}(v_i)$ of $\mathcal{G}_\infty$ are disjoint subsets of $B_{(4C+1)a}(v)$. This entails $m\tfrac{1}{c}\left(\tfrac{a}{3}\right)^d\le c(4C+1)^da^d$, whence $m\le D$.

\begin{figure}[htbp]
	\centering
	\includegraphics[scale=0.7]{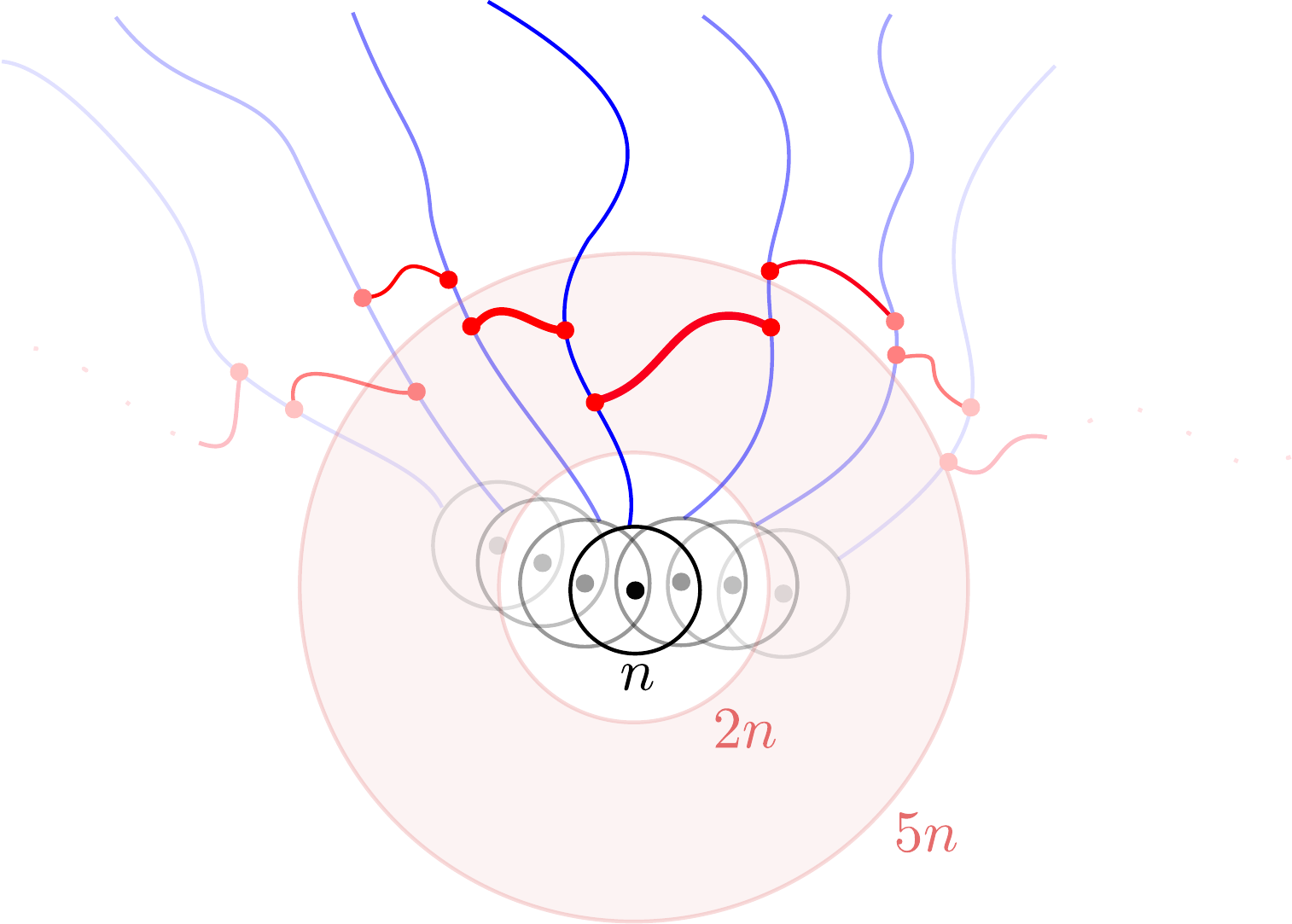}
	\caption{An open path in $\eta$ produces an open path in $\omega$ as a consequence of the gluing effect of the uniqueness zones between scale $2n$ and $5n$.}\label{fig:1}
\end{figure}

\textbf{We now work  on $\mathcal G_k$} and export the finite-size criterion constructed on $\mathcal G_\infty$ to $\mathcal G_k$, for $k$ large. This will show that $p$ is also supercritical on these graphs. 
For $k\in\mathbb{N}$ and $v\in V(\mathcal{G}_k)$, we define $\mathsf{E}_{n,k}(v)$ by the same conditions as above, but in $\mathcal{G}_k$ rather than $\mathcal{G}_\infty$. 
Given a percolation configuration $\omega$ on $\mathcal{G}_k$, we associate a site percolation configuration $\eta$ on $V_k$ by declaring $v\in V_k$ to be open if and only if $\mathsf{E}_{n,k}(v)$ holds.
This process enjoys the following two properties:
  \begin{itemize}
  \item if there is an infinite open path for $\eta$, then there is one for $\omega$ (see Figure~\ref{fig:1}),
  \item $V_k$ being endowed with its graph-structure of $(a,Ca)$-net, the process $\eta$ is $\tfrac{20n}{Ca}$-independent.
  \end{itemize}
The second property follows from Lemma~\ref{lem:geom}.

By definition of $a$, we have $\tfrac{20n}{Ca}=80$. By taking $k$ large enough, we may assume that the $10n$-ball of $\mathcal{G}_k$ is isomorphic to that of $\mathcal{G}_\infty$.  Therefore, $\mathrm{P_p}(\mathsf{E}_{n,k})=\mathrm{P}_p(\mathsf{E}_n)\ge q(D,80)$. By definition of $q(D,80)$ and because $p_c^\mathrm{site}(V_k) \le \tfrac{3}{4}$, the process $\eta$ yields an infinite cluster with positive probability, hence so is the case for $\omega$. We thus get $p_c(\mathcal{G}_k)\ge p$, as desired.
\end{proof}

\section{Monotonicity of $p_c$ revisited}\label{sec:monot-p_c-revis}

In this section, we revisit monotonicity of $p_c$ relative to the quotient operation, i.e. relative to covering maps. The original result is in \cite[Theorem 1]{MR1423907}, see also \cite{martseve2019monotone}. Our proof below is exactly the same. However, we state the proposition in a more general way that emphasises exactly which properties of the ``covering map'' $\pi$ are required for the argument to hold.

\begin{proposition}\label{prop:mono}
Let $\mathcal{G}$ and $\mathcal{G}'$ be two locally finite graphs with countably many vertices. Let $\pi : V(\mathcal{G})\to V(\mathcal{G}')$ be a surjective map.
Assume that for every $u\in V(\mathcal{G})$, for every $y$ neighbour of $\pi(u)$, there is at least one neighbour $v$ of $u$ such that $\pi(v)=y$.

Then, $p_c^\mathrm{site}(\mathcal{G'})\leq p_c^\mathrm{site}(\mathcal{G})$.
\end{proposition}

\begin{remark}If $\mathcal{G}'$ is connected and $\mathcal{G}$ is non-empty, then any $\pi : V(\mathcal{G})\to V(\mathcal{G}')$ satisfying the main assumption of the lemma is automatically surjective. Also note that $\pi$ is not assumed to be a graph homomorphism.
\end{remark}

\begin{proof}
Let $p>p_c^\mathrm{site}(\mathcal{G}')$. We can thus pick $o\in V(\mathcal{G}')$ such that $\mathbb{P}_p(o \leftrightarrow \infty)>0$. Since $\pi$ is onto, we can pick $\tilde{o}\in V(\mathcal{G})$ such that $\pi(\tilde{o})=o$. It suffices to prove that $\mathbb{P}_p(\tilde{o} \leftrightarrow \infty)\ge \mathbb{P}_p(o \leftrightarrow \infty)$, as this implies $p\ge p_c^\mathrm{site}(\mathcal{G})$.

We consider an exploration $(O_n,C_n)$ of the cluster of $o$ in $\mathcal{G}$, where $O_n$ and $C_n$ represent the open and closed vertices revealed up to  step $n$. The exploration starts with $(O_0,C_0)=(\{o\},\emptyset)$ if the origin $o$ is open, and we set $(O_0,C_0)=(\emptyset,\{o\})$ if the origin is closed. At  step $n$,  we pick (if it exists) an edge $(x_n,y_n)$  where   $x_n\in O_n$ is an explored  open vertex  and  $y_n\in V(\mathcal G)\setminus (O_n\cup C_n)$  is an unexplored vertex. The exploration stops if the edge $(x_n,y_n)$ does not exist.
If the exploration does not stop, we set
\begin{equation}
  \label{eq:5}
  (O_{n+1},C_{n+1})=
  \begin{cases}
    (O_n\cup\{y_n\},C_n)&\text{ if $y_n$ is open,}\\
    (O_n,C_n\cup\{y_n\})&\text{ if $y_n$ is closed,}
  \end{cases}
\end{equation}
and move to the next step.

This exploration can be lifted to define an exploration $(O_n',C_n')$ of (a subset of) the cluster of $o'$ in $\mathcal G'$. To do so, we start with $(O_0',C_0')=(\{o'\},\emptyset)$ if the origin $o$ in $\mathcal G$ is open, and we set $(O_0,C_0)=(\emptyset,\{o'\})$ otherwise. At step $n$, when we pick an edge $(x_n,y_n)$ in $\mathcal{G}$,  we can choose an edge $(u_n,v_n)$ in $\mathcal G'$ such that  $\pi(u_n)=x_n$, $\pi(v_n)=y_n$,  $u_n\in O_n$ and  $y_n\in V(\mathcal{G}') \setminus (O_n'\cup C_n')$, and then we define
\begin{equation}
  \label{eq:6}
  (O_{n+1}',C_{n+1}')=
  \begin{cases}
    (O_n'\cup\{v_n\},C_n')&\text{ if $y_n$ is open},\\
    (O_n',C_n'\cup\{v_n\})&\text{ if $y_n$ is closed}.
  \end{cases}
\end{equation}
Such a choice of $(u_n,v_n)$ is always possible due to our assumption on $\pi$.
If the  exploration in $\mathcal G$  never stops (which corresponds to the cluster of $o$ being infinite), then the lifted exploration does not stop either, which implies that the cluster of $o'$ is also infinite. 
\end{proof}

\newcommand{\etalchar}[1]{$^{#1}$}

\end{document}